\documentclass[a4paper,12pt]{amsart}

\usepackage[utf8]{inputenc}
\usepackage{amsfonts,amssymb}
\usepackage{amsmath}
\usepackage{graphicx}

\usepackage{comment}
\usepackage[usenames,dvipsnames]{xcolor}

\usepackage{a4wide}

\usepackage{empheq}
\numberwithin{equation}{section}

\usepackage{tikz}
\usetikzlibrary{decorations.markings}
\tikzset{->-/.style={decoration={
  markings,
  mark=at position 0.5 with {\arrow{>}}},postaction={decorate}}}

\usepackage{amsthm}
\theoremstyle{plain}
\newtheorem*{theorem*}{Theorem}
\newtheorem{theorem}{Theorem}[section]
\newtheorem{lemma}[theorem]{Lemma}
\newtheorem{corollary}[theorem]{Corollary}

\newtheorem{proposition}[theorem]{Proposition}

\newtheorem*{lemma*}{Lemma}
\newtheorem*{question*}{Question}

\theoremstyle{definition}
\newtheorem{definition}[theorem]{Definition}
\newtheorem{example}[theorem]{Example}
\theoremstyle{remark}
\newtheorem{remark}[theorem]{Remark}

\newcommand{\R}{\mathbb{R}}
\newcommand{\N}{\mathbb{N}}
\newcommand{\Z}{\mathbb{Z}}

\newcommand{\defeq}{\mathrel{\mathop{:}}=}
\newcommand{\eqdef}{=\mathrel{\mathop{:}}}
\newcommand{\abs}[1]{\lvert #1 \rvert}

\newcommand{\lk}{\operatorname{lk}}

\newcommand{\dlk}{{\lk}{\downarrow}}

\newcommand{\CAT}{\operatorname{CAT}}
\newcommand{\F}{\text{F}}
\newcommand{\FP}{\text{FP}}
\newcommand{\id}{\operatorname{id}}

\newcommand{\calG}{\Gamma}
\newcommand{\frakh}{\mathfrak{h}}

\numberwithin{equation}{section}

\newcommand{\picangle}{60}
\newcommand{\loopangle}{40}

\begin{document}

\title{The Basilica Thompson group is not finitely presented}
\date{\today}


\author[S.~Witzel]{Stefan Witzel}
\address{Department of Mathematics, Bielefeld University, PO Box 100131, 33501 Bielefeld, Germany}
\email{switzel@math.uni-bielefeld.de}

\author[M.~C.~B.~Zaremsky]{Matthew C.~B.~Zaremsky}
\address{Department of Mathematical Sciences, Binghamton University, Binghamton, NY 13902}
\email{zaremsky@math.binghamton.edu}

\begin{abstract}
 We show that the Basilica Thompson group introduced by Belk and Forrest is not finitely presented, and in fact is not of type $\FP_2$. The proof involves developing techniques for proving non-simple connectedness of certain subcomplexes of $\CAT(0)$ cube complexes.
\end{abstract}

\maketitle
\thispagestyle{empty}


\section*{Introduction}

J. Belk and B. Forrest \cite{belk15} introduced the \emph{Basilica Thompson group} $T_B$ (defined in Definition~\ref{def:T_B} below). They showed that it is virtually simple, generated by four elements, and is a sub- as well as a supergroup of Thompson's group $T$. The question of whether it is finitely presented, however, remained open. In this paper we prove:

\begin{theorem*}
 $T_B$ is not finitely presented.
\end{theorem*}

The Basilica Thompson group is an example of a \emph{rearrangement group of a fractal} as defined by Belk and Forrest in \cite{belk15a}. These groups arise from \emph{edge replacement systems} and act naturally on self-similar spaces, for example Julia sets. They generalize \emph{Thompson's groups} $F$, $T$ and $V$. For certain rearrangement groups Belk and Forrest proved that they are of \emph{type $\F_\infty$}, confirming the ``expected'' behavior for relatives of Thompson's groups. As for the classical Thompson's groups, the proof relies on studying the action on an associated $\CAT(0)$ cube complex. The reason that (non-)finite presentability of $T_B$ remained open is that the known local methods are not suited to prove the kind of negative connectivity statement needed. Our proof of the theorem involves a global analysis.

We now know of (virtually) simple rearrangement groups with two extremal finiteness properties: of type $\F_1$ but not of type $\F_2$, and of type $\F_\infty$. If the intermediate finiteness properties could also be achieved, e.g.\ of type $\F_n$ but not $\F_{n+1}$ for any $n$, this would provide an infinite family of pairwise non-quasi-isometric simple groups (the only other known examples of this kind are hyperbolic Kac--Moody groups \cite{caprace09}). We therefore ask:

\begin{question*}
 Is there a rearrangement group that is virtually simple, finitely presented, but not of type $F_\infty$?
\end{question*}

\subsection*{Acknowledgments} We are grateful to Jim Belk for posing this problem to us, to Belk and Brad Forrest for helpful discussions, and to the Binghamton University math department for its hospitality while the first named author was visiting the second named author. The first named author also gratefully acknowledges support by the DFG through the project WI 4079/2 and through the SFB 701.


\section{Background}\label{sec:background}

In this section we recall the necessary background material from Sections 1, 2 and 3 of \cite{belk15a} on \emph{edge replacement systems}, and the resulting groups and complexes. At the end we also recall some background on discrete Morse theory, and on $\CAT(0)$ cube complexes.

\subsection{Edge replacement systems}

We will consider finite directed graphs $G$. The notation $V(G)$ and $E(G)$ will always mean the vertex and edges sets of $G$.

\begin{definition}[Edge replacement rule]\label{def:replacement_rule}
 An \emph{(edge) replacement rule} is a pair $e\to R$, for $e$ a non-loop directed edge, say with initial vertex $v$ and terminal vertex $w$, and $R$ a finite directed graph with $v,w\in V(R)$.
\end{definition}

\begin{definition}[Edge replacement system]\label{def:replacement_system}
 An \emph{(edge) replacement system} is a pair $(G_0,e\to R)$, for $G_0$ a finite directed graph, called the \emph{base graph}, and $e\to R$ a replacement rule. We will always assume $(G_0,e\to R)$ is \emph{expanding} (see \cite[Definition~1.8]{belk15a}), i.e.\ $G_0$ has no isolated vertices, $v$ and $w$ do not share an edge in $R$, and $|V(R)|\ge 3$ and $|E(R)|\ge 2$.
\end{definition}

Applying a replacement rule to the edge $\varepsilon\in E(G)$ of a graph $G$ amounts to removing the edge $\varepsilon$ and replacing it by $R$. Note that we do allow $\varepsilon$ to be a loop. We denote by $G\lhd \varepsilon$ the graph obtained from $G$ by replacing the edge $\varepsilon$ by $R$ in this way. We call $G\lhd \varepsilon$ a \emph{simple expansion} of $G$. Any graph obtained from $G$ by a finite sequence of simple expansions is called an \emph{expansion} of $G$. The reverse of a (simple) expansion is called a \emph{(simple) contraction}. We will ``address'' edges and vertices of expansions by concatenating addresses of edges and vertices from $G$ and $R$; for example, if the edges of $R$ are called $1,2,3$ then the new edges in the expansion $G\lhd \varepsilon$ are called $\varepsilon1,\varepsilon2,\varepsilon3$, and the old edges retain their addresses from $G$. See Example~\ref{ex:T_B} for pictures.

\begin{definition}[Limit space]\label{def:limit space}
 Let $\mathcal{R}=(G_0,e\to R)$ be an edge replacement system. For each $n\in\N$ let $G_n$ be the result of applying the replacement rule to each edge of $G_{n-1}$ once. Let $\Omega\defeq E(G_0)\times E(R)^\infty$ be the set of sequences of edges, with leading edge from $G_0$ and all others from $R$. Declare that two such sequences $\varepsilon_0 \varepsilon_1 \cdots$ and $\varepsilon_0' \varepsilon_1' \cdots$ are \emph{equivalent} if for all $n$ the edges of $G_n$ with ``addresses'' $\varepsilon_0\cdots \varepsilon_n$ and $\varepsilon_0'\cdots \varepsilon_n'$ share a vertex. Define the \emph{limit space} $X$ for $\mathcal{R}$ to be the space of equivalence classes $[\varepsilon_0 \varepsilon_1 \cdots]$ of such sequences.
\end{definition}

The limit space $X$ of a replacement system $\mathcal{R}$ is compact and metrizable (\cite[Theorem~1.24]{belk15a}). Our groups of interest are certain groups of homeomorphisms of such limit spaces, which is the subject of the next subsection.

\subsection{Rearrangement groups}

We now define \emph{rearrangement groups}, which are certain groups of homeomorphisms of limit spaces.

\begin{definition}[Cell]
 Let $X$ be the limit space of a replacement system $\mathcal{R}=(G_0,e\to R)$. Let $\varepsilon = \varepsilon_0\cdots\varepsilon_n$ be an edge of some expansion of $G_0$. The \emph{cell} $C(\varepsilon) \subseteq X$ is the subspace consisting of equivalence classes of sequences representable by a sequence with $\varepsilon_0\cdots\varepsilon_n$ as a prefix.
\end{definition}

\begin{definition}[Canonical homeomorphism]
 Let $C(\varepsilon)$ and $C(\varepsilon')$ be two cells such that $\varepsilon$ and $\varepsilon'$ are either both loops or both non-loops. The \emph{canonical homeomorphism} $\Phi \colon C(\varepsilon)\to C(\varepsilon')$ is the map defined via the prefix replacement rule
 $$\Phi([\varepsilon\zeta_1\zeta_2\cdots])\defeq [\varepsilon'\zeta_1\zeta_2\cdots]\text{.}$$
\end{definition}

\begin{definition}[Rearrangement]\label{def:rearrangement}
 Let $X$ be the limit space of a replacement system $\mathcal{R}=(G_0,e\to R)$. A homeomorphism $f\colon X\to X$ is called a \emph{rearrangement} if there exist finitely many cells $C(\varepsilon_1),\dots,C(\varepsilon_n)$ such that the cells cover $X$, have pairwise disjoint \emph{interiors} (defined in \cite[Section~1.3]{belk15a}), and such that each restriction $f|_{C(\varepsilon_i)}$ is a canonical homeomorphism.
\end{definition}

\begin{definition}[Rearrangement group]
 The rearrangements of $X$ form a group \cite[Proposition~1.15]{belk15a}, called the \emph{rearrangement group} $\calG$ of $X$.
\end{definition}

\begin{definition}[Graph pair diagram]\label{def:graph_pair_diagram}
 Let $f\colon X\to X$ be a rearrangement. A \emph{graph pair diagram} for $f$ is a triple $(E_-,E_+,\varphi)$, where $E_\pm$ are expansions of $G_0$ and $\varphi \colon E_- \to E_+$ is a graph isomorphism, such that for every edge $\varepsilon$ of $E_-$ the restriction of $f$ to $C(\varepsilon)$ is a canonical homeomorphism from $C(\varepsilon)$ to $C(\varphi(\varepsilon))$. The idea is that, even though $f$ is a homeomorphism of $X$, it can already be realized at some finite expansion stage. In a graph pair diagram $(E_-,E_+,\varphi)$, we call $E_-$ the \emph{domain graph} and $E_+$ the \emph{range graph}.
\end{definition}

\begin{example}[The Basilica rewriting system]\label{ex:T_B}
Consider the replacement rule

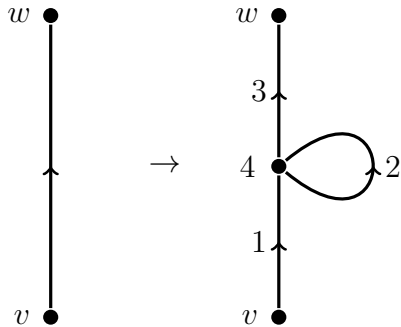
\begin{center}
\begin{tikzpicture}[scale=2, very thick]
\begin{scope}[every node/.style={circle, fill, inner sep=2pt}]
\node[label=left:{$v$}] (lv) at (0,0) {};
\node[label=left:{$w$}] (lw) at (0,2) {};

\node[label=left:{$v$}] (v) at (1.5,0) {};
\node[label=left:{$w$}] (w) at (1.5,2) {};
\node[label=left:{$4$}] (4) at (1.5,1) {};
\end{scope}
\draw[->-] (lv) to (lw);

\node at (0.75,1) {$\to$};

\draw[->-] (v) to[left] node{$1$} (4);
\draw[->-] (4) to[out=-\loopangle, in=\loopangle, min distance=1cm, looseness=5,right] node{$2$} (4);
\draw[->-] (4) to[left] node{$3$} (w);
\end{tikzpicture}
\end{center}

where $E(R) = \{1,2,3\}$ and $V(R) = \{4\}$. If $G_0$ is the graph

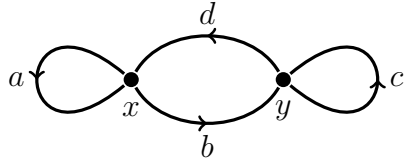
\begin{center}
\begin{tikzpicture}[scale=2, very thick]
\begin{scope}[every node/.style={circle, fill, inner sep=2pt}]
\node[label=below:{$x$}] (x) at (0,0) {};
\node[label=below:{$y$}] (y) at (1,0) {};
\end{scope}
\begin{scope}
\path (x) edge[->-,out=180-\loopangle, in = 180+\loopangle, min distance = 1cm, looseness=5, left] node {$a$} (x);
\path (y) edge[->-,out=0-\loopangle, in = 0+\loopangle, min distance = 1cm, looseness=5, right] node {$c$} (y);
\path (x) edge[->-,out=0-\picangle, in = 180+\picangle, below] node {$b$} (y);
\path (y) edge[->-,out=180-\picangle, in = 0+\picangle, above] node {$d$} (x);
\end{scope}
\end{tikzpicture}
\end{center}

then for example the graphs

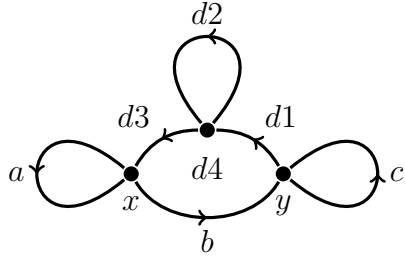
\begin{center}
\begin{tikzpicture}[scale=2, very thick,none/.style={text opacity=0,opacity=0,fill opacity=0}]
\begin{scope}[every node/.style={circle, fill, inner sep=2pt}]
\node[label=below:{$x$}] (x) at (0,0) {};
\node[label=below:{$y$}] (y) at (1,0) {};
\draw[none] (y) to [out=180-\picangle, in=0+\picangle]  coordinate[pos=0.5] (xy) (x);
\node[label=below:{$d4$}] (d4) at (xy) {};
\end{scope}
\begin{scope}
\path (x) edge[->-,out=180-\loopangle, in = 180+\loopangle, min distance = 1cm, looseness=5, left] node {$a$} (x);
\path (y) edge[->-,out=0-\loopangle, in = 0+\loopangle, min distance = 1cm, looseness=5, right] node {$c$} (y);
\path (d4) edge[->-,out=90-\loopangle, in = 90+\loopangle, min distance = 1cm, looseness=5, above] node {$d2$} (d4);
\path (x) edge[->-,out=0-\picangle, in = 180+\picangle, below] node {$b$} (y);
\path (y) edge[->-,out=180-\picangle, in = 0, above right] node {$d1$} (d4);
\path (d4) edge[->-,out=180, in = 0+\picangle, above left] node {$d3$} (x);
\end{scope}
\end{tikzpicture}
\end{center}
and
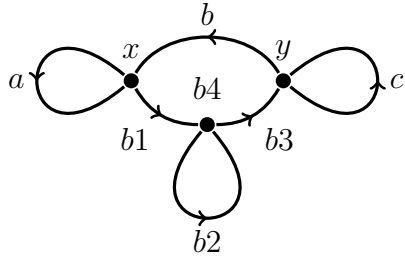
\begin{center}
\begin{center}
\begin{tikzpicture}[scale=2, very thick,none/.style={text opacity=0,opacity=0,fill opacity=0}]
\begin{scope}[every node/.style={circle, fill, inner sep=2pt}]
\node[label=above:{$x$}] (x) at (0,0) {};
\node[label=above:{$y$}] (y) at (1,0) {};
\draw[none] (x) to [out=0-\picangle, in=180+\picangle]  coordinate[pos=0.5] (xy) (y);
\node[label=above:{$b4$}] (b4) at (xy) {};
\end{scope}
\begin{scope}
\path (x) edge[->-,out=180-\loopangle, in = 180+\loopangle, min distance = 1cm, looseness=5, left] node {$a$} (x);
\path (y) edge[->-,out=0-\loopangle, in = 0+\loopangle, min distance = 1cm, looseness=5, right] node {$c$} (y);
\path (b4) edge[->-,out=-90-\loopangle, in = -90+\loopangle, min distance = 1cm, looseness=5, below] node {$b2$} (b4);
\path (y) edge[->-,out=180-\picangle, in = 0+\picangle, above] node {$b$} (x);
\path (x) edge[->-,out=0-\picangle, in = 180, below left] node {$b1$} (b4);
\path (b4) edge[->-,out=0, in = 180+\picangle, below right] node {$b3$} (y);
\end{scope}
\end{tikzpicture}
\end{center}
\end{center}

are expansions of $G_0$, with the addresses given for the edges and vertices. For $X$ the limit space of the rearrangement system, the map $X \to X$ given piecewise by the canonical homeomorphisms
\begin{align*}
a* &\mapsto a*\\
b* &\mapsto b1*\\
c* &\mapsto b2*\\
d1* &\mapsto b3*\\
d2* &\mapsto c*\\
d3* &\mapsto d*
\end{align*}
is a rearrangement. The graph pair diagram for this map consists of the above graphs, with the isomorphism $\varphi$ given by erasing the ``$*$'''s in these canonical homeomorphisms.
\end{example}

\begin{definition}\label{def:T_B}
 The \emph{Basilica Thompson group} $T_B$ is the rearrangement group of the rewriting system in Example~\ref{ex:T_B}.
\end{definition}

\subsection{Cube complexes for rearrangement groups}

In this subsection, we recall the $\CAT(0)$ cube complex on which a rearrangement group $\calG$ acts. The replacement rule $e\to R$ will be fixed throughout. Given a base graph $G_0$, we will denote the cube complex by $K(G_0,e\to R)$. In \cite{belk15a} it was denoted $K(\calG)$, but for our purposes it is important to keep track of the base graph $G_0$ used, and less important to keep track of the group $\calG$. For any choice of base graph $G_0$, denote by $X(G_0)$ the limit space of the replacement system $(G_0,e\to R)$. We extend the definition of rearrangement from Definition~\ref{def:rearrangement} as follows:

\begin{definition}[Rearrangement]
 Let $G_0$ and $G$ be two choices of base graph, so we have limit spaces $X(G_0)$ and $X(G)$. A homeomorphism $f\colon X(G_0)\to X(G)$ is called a \emph{rearrangement} if there exist finitely many cells $C(\varepsilon_1),\dots,C(\varepsilon_n)$ such that the cells cover $X(G_0)$, have pairwise disjoint interiors, and such that each restriction $f|_{C(\varepsilon_i)}$ is a canonical homeomorphism.
\end{definition}

The category whose objects are the $X(G)$ and whose morphisms are rearrangements is a groupoid. Note that depending on the choices of $G_0$ and $G$, a rearrangement $X(G_0)\to X(G)$ might not exist, thus the groupoid is not connected.

We can also extend the definition of graph pair diagram from Definition~\ref{def:graph_pair_diagram}:

\begin{definition}[Graph pair diagram]
 Let $f\colon X(G_0)\to X(G)$ be a rearrangement. A \emph{graph pair diagram} for $f$ is a triple $(E_-,E_+,\varphi)$ where $E_-$ is an expansion of $G_0$, $E_+$ is an expansion of $G$ and $\varphi \colon E_- \to E_+$ is a graph isomorphism, such that for every edge $\varepsilon$ of $E_-$ the restriction of $f$ to $C(\varepsilon)$ is a canonical homeomorphism from $C(\varepsilon)$ to $C(\varphi(\varepsilon))$.
\end{definition}

Rearrangements can be decomposed into certain fundamental rearrangements.

\begin{definition}[Special rearrangements]
 If $G$ is an expansion of $G_0$, there is a canonical rearrangement $\varphi \colon X(G_0) \to X(G)$ with diagram $(G,G,\id)$. We call it an \emph{expansion rearrangement}. Its inverse is a \emph{contraction rearrangement}. If the expansion was simple, we also say that the expansion/contraction rearrangement is \emph{simple}. If $\varphi \colon G_0 \to G$ is a graph isomorphism then the diagram $(G_0,G,\varphi)$ represents a rearrangement $X(G_0) \to X(G)$, also denoted $\varphi$, which is called a \emph{base isomorphism}.
\end{definition}

Thus any rearrangement is a product of an expansion rearrangement, a base isomorphism, and a contraction rearrangement.

We are now approaching the definition of the cube complex $K(G_0,e\to R)$. The fundamental objects here are rearrangements with a fixed domain $X(G_0)$.

\begin{definition}[Range equivalence, expansion/contraction]
 Let $f\colon X(G_0) \to X(G_1)$ and $g\colon X(G_0) \to X(G_2)$ be two arrangements. We say that $f$ and $g$ are \emph{range equivalent} if there is a base isomorphism $\varphi \colon X(G_1) \to X(G_2)$ such that $\varphi \circ f = g$. We write $[f]$ for the range equivalence class of $f$. We say that $g$ is a \emph{(simple) expansion} of $f$ if there is a (simple) expansion rearrangement $\varphi \colon X(G_1) \to X(G_2)$ such that $\varphi \circ f = g$. In that case $f$ is a \emph{(simple) contraction} of $g$. We also apply these notions to $[f]$ and $[g]$. In special cases we will need notation for this. If $S = \{\varepsilon_1, \ldots, \varepsilon_k\} \subseteq E(G_1)$ is the set of edges of $G_1$ such $G_2 =  G_1 \lhd \varepsilon_1 \lhd \ldots \varepsilon_k$ (the order does not matter) then we write $\varphi = \Delta_S$. Conversely if $T = \{R_1, \ldots, R_k\}$ are the replacements for the $\varepsilon_i$ then we write $\varphi^{-1} = \nabla_T$. If such an $S$ or $T$ has only one element, we may omit the set braces from the notation, and so write things like $\Delta_\varepsilon$ and $\nabla_{R}$. In that case we may specify the subgraph $R$ of $G_2$ by listing its edges, writing for example $\nabla_{\varepsilon1,\varepsilon2,\varepsilon3}$.
\end{definition}



\begin{definition}[The cube complex]
 Let $K(G_0,e\to R)$ be the cube complex defined as follows. There is a $0$-cube for every range equivalence class $[f]$ of rearrangements with domain $X(G_0)$. For each $0$-cube $[f]$, say represented by the rearrangement $f\colon X(G_0)\to X(G)$, and for each $S \subseteq E(G)$, there is an $\abs{S}$-dimensional cube whose $0$-subcube set is $\{[\Delta_T\circ f]\mid T\subseteq S\}$.
\end{definition}

For example, if $f \colon X(G_0) \to X(G)$ is a rearrangement and $\varepsilon$ is an edge of $G$ then $[f]$ and $[\Delta_\varepsilon \circ f]$ span a $1$-cube. Belk and Forrest proved that $K(G_0,e\to R)$ is a $\CAT(0)$ cube complex:

\begin{proposition}[{\cite[Proposition~3.33, Corollary~3.24]{belk15a}}]
 The complex $K(G_0,e\to R)$ is contractible, and in fact is $\CAT(0)$.
\end{proposition}

\begin{definition}[Rank]\label{def:rank}
 The \emph{rank} of a rearrangement $f \colon X(G_0) \to X(G)$ is $\mu(f)\defeq |E(G)|$. Since range equivalent rearrangements have the same rank, we can also define the rank of an equivalence class $\mu([f])\defeq \mu(f)$.
\end{definition}

It is easy to see (and explained in \cite{belk15a}) that $\mu$ extends to a Morse function on the cube complex $K(G_0,e\to R)$ (for any graph $G_0$ and replacement system $e\to R$). We review the relevant Morse theoretic concepts in the next subsection.

\subsection{Morse theory}\label{sec:morse}

\begin{definition}[Morse function]
 Let $X$ be a cube complex and let $h\colon X\to \R$ be a map. We call $h$ a \emph{Morse function} if the following properties hold:
 \begin{enumerate}
  \item The image $h(X^{(0)})$ is discrete in $\R$.
  \item For any cube $c$, the restriction of $h$ to $c$ is an affine function $h|_c \colon c\to\R$.
  \item For any cube $c$ of positive dimension, the restriction of $h$ to $c$ is non-constant.
 \end{enumerate}
\end{definition}

Here when we say that $h|_c \colon c\to \R$ is an affine function, we are viewing $c$ as $[-1,1]^{\dim(c)}$ with the usual affine structure.

Given a cube complex $X$ with a Morse function $h$, for any $m\in\R$ we denote by $X_m$ the subcomplex of $X$ supported on those $0$-cubes $x$ with $h(x)\le m$.
%
%
%
%
The \emph{descending link} $\dlk x$ of a $0$-cube $x$ is the link of $x$ in the sublevel set $X_{h(x)}$. That is, the descending link is the subcomplex of the link $\lk x$ supported on those $0$-simplices along which $h$ is decreasing. Our Morse Lemma (which is a special case of \cite[Corollary~2.6]{bestvina97}) is as follows:

\begin{lemma}[Morse Lemma]\label{lem:morse}
 Let $X$ be a cube complex and $h\colon X\to \R$ a Morse function. Suppose $m\le M$ are real numbers, and that for all $0$-cubes $x\in X^{(0)}$ with $m<h(x)\le M$ the descending link $\dlk x$ is $(n-1)$-acyclic. Then the inclusion $X_m \to X_M$ induces an isomorphism in $H_k$ for $k\le n-1$, and an epimorphism in $H_n$.
\end{lemma}

The connection between Morse theory and finite presentability is made as follows. First recall that if a group is finitely presented then it satisfies the homological finiteness property of being of type $\FP_2$. To show that $T_B$ is not of type $\FP_2$ we will use the following criterion.

\begin{lemma}\label{lem:fin_iff_conn}
 Let a group $\calG$ act with finite stabilizers on a $1$-acyclic cube complex $X$. Let $h\colon X\to\R$ be a Morse function on $X$, and suppose that each $X_m$ is $\calG$-invariant and cocompact. Suppose there exists $N\in\N$ such that for any $x\in X^{(0)}$ with $h(x)>N$, the descending link of $x$ is connected. Then $\calG$ is of type $\FP_2$ if and only if $X_m$ is $1$-acyclic for some $m \ge N$.
\end{lemma}

\begin{proof}
 If some $X_m$ is $1$-acyclic then it is of type $\FP_2$ by \cite[Proposition~1.1]{brown87}. In the converse direction, the Morse Lemma implies that the maps $H_1(X_N \to X_m)$ are all surjective for $m \ge N$. Thus if $X_m$ is not $1$-acyclic for any $m$, then the map $H_1(X_n \to X_m)$ cannot be trivial for any $N \le n \le m$ (since the previous map is surjective and factors through this one). Hence the system $H_1(X_n), n \ge N$ is not essentially trivial, and \cite[Theorem~2.2]{brown87} implies that $G$ is not of type $\FP_2$.
\end{proof}

We close the section with a strengthening of the well known nerve lemma. We will only need the case $n=1$, where it is a homological version of \cite[Lemma~6.2]{witzel15}, but having the question from the introduction in mind, we prove the general case. The main point is the surjective morphism at the end of the statement.

\begin{proposition}[Strong nerve lemma]\label{prop:nerve}
 Let a CW complex $X$ be covered by subcomplexes $(X_i)_{i \in I}$, let $L$ be the nerve of the cover and let $n \in \N$ be arbitrary. Assume that if $X_{i_1} \cap \ldots \cap X_{i_r} \ne \emptyset$ then
 \[
 H_q(X_{i_1} \cap \ldots \cap X_{i_r}) = H_q(\text{pt.})
 \]
 for $1 \le r \le n-q$. Then $H_k(X) \cong H_k(L)$ for $k < n$ and there is a surjective morphism $H_n(X) \to H_n(L)$.
\end{proposition}

\begin{proof}
 The nerve is the simplicial complex of those $T \subseteq I$ with $X_T \defeq \bigcap_{i \in T} X_i \ne \emptyset$. It is equipped with the coefficient systems $\frakh_q \colon T \to H_q(X_T)$. We use the spectral sequence from \cite[Theorem~E.3.2]{davis08}:
\[
E^2_{p,q} = H_p(L;\frakh_q) \Rightarrow H_*(X)\text{.}
\]
Our conditions ensure that $\frakh_q(T) = 0$ provided $T$ has dimension at most $n-q-1$ for $q > 0$.
This means that
\begin{equation}
\label{eq:spec_seq}
E^2_{p,q} =
\left\{
\begin{array}{ll}
H_p(L; \frakh_0) & q = 0 \text{ and } 0 \le p \le n\\
0 & 1 \le q \le n - p - 1\text{.}
\end{array}
\right.
\end{equation}
Since the spectral sequence lies in the first quadrant, the region $p + q \le n - 1$ of \eqref{eq:spec_seq} remains stable, which tells us that $H_p(X) \cong H_p(L;\frakh_0)$ for $0 \le p \le n-1$ and that there is an epimorphism $H_n(X) \to H_n(L;\frakh_0)$.

Thus it remains to show that $H_p(L;\frakh_0) \cong H_p(L)$ for $0 \le p \le n-1$ and to produce an epimorphism $H_n(L;\frakh_0) \to H_n(L)$. We claim that the map of coefficient systems $\frakh_0 \to \Z$ (here the codomain is the coefficient system that is constant $\Z$) induced by contracting all $X_T$ to a point works. To spell this out let $(C_p)_{p \in \N}$ be the chain complex of $L$, i.e.\ $C_p \cong \bigoplus_{T \in \Sigma_p} \Z$ where $\Sigma_p$ is the set of $p$-simplices of $L$. We think of the generator of the summand corresponding to $T$ as $X_T$. Let $\tilde{C}_p$ be the chain complex of $L$ with coefficients in $\frakh_0$, i.e.\ $\tilde{C}_p \cong \bigoplus_{T \in \Sigma_p} H_0(X_T)$. Our assumptions ensure that $H_0(X_T) = \Z$ for $T$ of dimension at most $n-1$, so we can identify $C_p$ with $\tilde{C}_p$ for $p \le n-1$. It follows that $H_p(L) \cong H_p(L;\frakh_0)$ for $p \le n-2$. It is also easy to see that the images of $C_n \stackrel{\partial}{\to} C_{n-1}$ and of $\tilde{C}_n \stackrel{\partial}{\to} C_{n-1}$ are the same so $H_{n-1}(L) \cong H_{n-1}(L;\frakh_0)$.

In general (to cover degree $n$), we look at the map of chain complexes $\psi_* \colon \tilde{C}_* \to C_*$ taking each component of $X_T$ to $X_T$. The key point is that this is actually a map of chain complexes, in particular that the square

\vspace{-\medskipamount}
\begin{center}
\begin{tikzpicture}[scale=2]
\node (ul) at (0,0) {$\tilde{C}_{n+1}$};
\node (ur) at (1,0) {$\tilde{C}_{n}$};
\node (ll) at (0,-1) {$C_{n+1}$};
\node (lr) at (1,-1) {$C_{n}$};
\path (ul) edge[->,above] node {$\partial$} (ur);
\path (ul) edge[->,left] node {$\psi_{n+1}$} (ll);
\path (ur) edge[->, left] node {$\psi_n$} (lr);
\path (ll) edge[->, above] node {$\partial$} (lr);
\end{tikzpicture}
\end{center}
\vspace{-\medskipamount}
commutes. This comes from the fact that if $f \colon X_U \to X_T$ is continuous (inclusion in our case) then the diagram
\vspace{-\medskipamount}
\begin{center}
\begin{tikzpicture}[xscale=2.5, yscale=2]
\node (ul) at (0,0) {$H_0(X_U)$};
\node (ur) at (1,0) {$H_0(X_T)$};
\node (ll) at (0,-1) {$H_0(\text{pt.})$};
\node (lr) at (1,-1) {$H_0(\text{pt.})$};
\path (ul) edge[->,above] node {$H_0(f)$} (ur);
\path (ul) edge[->] (ll);
\path (ur) edge[->] (lr);
\path (ll) edge[->] (lr);
\end{tikzpicture}
\end{center}
\vspace{-\medskipamount}
commutes. The rest is easy: any $c \in C_n$ is the image under $\psi_n$ of some $\tilde{c} \in \tilde{C}_n$ and if $\tilde{c} = \partial \tilde{d}$ then $c = \partial \psi_{n+1}(\tilde{d}) \eqdef \partial d$. Thus $\psi_*$ induces an epimorphism in degree $n$.
\end{proof}

\begin{corollary}
\label{cor:nerve}
Let a CW complex $X$ be covered by connected subcomplexes. If the nerve of the cover is not $1$-acyclic then neither is $X$.\qed
\end{corollary}

\subsection{$\CAT(0)$ cube complexes}\label{sec:cat0}

In this brief subsection we collect some terminology and results regarding $\CAT(0)$ cube complexes, all of which comes for example from \cite{haglund08}.

Let $X$ be a $\CAT(0)$ cube complex. A \emph{midcube} is the subset of $[-1,1]^n$ obtained by restricting some coordinate to $0$. The midcubes of $X$ form a new cube complex and each of its components is a \emph{hyperplane} in $X$. Each hyperplane is naturally embedded in $X$. A \emph{wall} of $X$ is a parallel class of oriented $1$-cubes of $X$. We will also use the term \emph{wall} to denote the cube complex whose $0$-cubes are those $1$-cubes, whose $1$-cubes are the corresponding $2$-cubes of $X$, and so forth. There is a natural $2$-to-$1$ correspondence between walls and hyperplanes and each wall is isomorphic to its corresponding hyperplane. If $c$ is a cube and $H$ is a wall containing a $1$-face of $c$ we say that $H$ \emph{cuts through} $c$. The $1$-faces of a $k$-cube lie in precisely $k$ walls. Given a hyperplane $H$, any subcomplex of $X$ supported on all the $0$-cubes in a connected component of $X\setminus H$ is called a \emph{half-space}. In fact each wall $H$ determines two half-spaces and we denote them $H^+$ and $H^-$.


\section{Proof}\label{sec:proof}

In this section we prove our main result, that the Basilica Thompson group $T_B$ is not finitely presented. The main technical lemma that is specific to the Basilica rewriting system is Lemma~\ref{lem:empty_sublevel}. The Lemmas~\ref{lem:nonempty_quarterspaces} and~\ref{lem:conn_halfspaces} should readily generalize to other setups. The rest of the section is general combinatorial topology which is conveniently phrased in the setup of cube complexes.

\begin{remark}
 We think that it would only take minor modifications (mostly to Lemma~\ref{lem:empty_sublevel}) to prove that, for example, the rearrangement groups for the rabbit Julia set and its variants are not finitely presented (see \cite{belk15a} for background). Thus a positive answer to the question posed initially will not come from Julia sets whose parameter is in a bulb of the Mandelbrot set adjacent to the main cardioid.
\end{remark}

Using Lemma~\ref{lem:fin_iff_conn} our main task is to show that certain sublevel sets are not $1$-acyclic. For that purpose we will use the following criterion.

\begin{lemma}\label{lem:two_walls}
 Let $X$ be a $\CAT(0)$ cube complex, let $Y$ be a subcomplex, and let $H_1,H_2$ be two walls in $X$. Assume that $H_1^\delta \cap Y$ and $H_2^\delta \cap Y$ are connected for $\delta\in\{+,-\}$, that $H_1^{\delta_1} \cap H_2^{\delta_2} \cap Y$ are all non-empty for $\delta_1,\delta_2\in\{+,-\}$, and that $H_1\cap H_2\cap Y$ is empty. Then $Y$ is not $1$-acyclic.
\end{lemma}

\begin{proof}
 The assumption that $H_1\cap H_2\cap Y$ is empty means that $Y$ contains no cube through which $H_1$ and $H_2$ both cut, which implies that $Y$ is covered by the subcomplexes $H_i^\delta \cap Y$, for $i\in\{1,2\}$ and $\delta\in\{+,-\}$. The nerve of this cover is $S^0 * S^0 \cong S^1$, by the assumption that the $H_1^{\delta_1} \cap H_2^{\delta_2} \cap Y$ are all non-empty. Since we are also assuming that each $H_i^\delta \cap Y$ is connected, and since the nerve is not $1$-acyclic, the result follows from Corollary~\ref{cor:nerve}.
\end{proof}

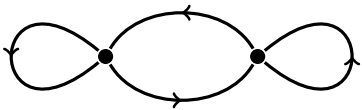
\begin{figure}[htb]
\begin{tikzpicture}[scale=2, very thick]
\begin{scope}[every node/.style={circle, fill, inner sep=2pt}]
\node[label=below:{}] (x) at (0,0) {};
\node[label=below:{}] (y) at (1,0) {};
\end{scope}
\begin{scope}
\path (x) edge[->-,out=180-\loopangle, in = 180+\loopangle, min distance = 1cm, looseness=5, left] node {} (x);
\path (y) edge[->-,out=0-\loopangle, in = 0+\loopangle, min distance = 1cm, looseness=5, right] node {} (y);
\path (x) edge[->-,out=0-\picangle, in = 180+\picangle, below] node {} (y);
\path (y) edge[->-,out=180-\picangle, in = 0+\picangle, above] node {} (x);
\end{scope}
\end{tikzpicture}
\caption{The graph $G_0$.}
\label{graph:G_0}
\end{figure}

Now consider the cube complex $K(G_0,e\to R)$, for $G_0$ the graph in Figure~\ref{graph:G_0} and $e\to R$ the Basilica replacement rule from Example~\ref{ex:T_B}. For the rest of this section, we will always be using this rewriting rule, so we omit it from the notation. Let $h\colon K(G_0)\to \R$ be the Morse function induced by the rank. To prove that $T_B$ is not of type $\FP_2$, Lemma~\ref{lem:fin_iff_conn} says it suffices to show that for $m$ large enough, no $K(G_0)_m$ is $1$-acylic. We will apply Lemma~\ref{lem:two_walls}, namely, we will find two walls in $K(G_0)$ satisfying the requirements for $Y = K(G_0)_m$.

First we claim that walls in $K(G_0)$ are isomorphic to cube complexes of the form $K(G)$, for appropriate $G$.

\begin{lemma}[Modeling walls]\label{lem:model_wall}
 Let $G_1$ be a graph, let $\varepsilon$ an edge of $G_1$, and let $G_2=G_1\lhd \varepsilon$. Let $f \colon X(G_0) \to X(G_1)$ and $g =\Delta_\varepsilon \circ f \colon X(G_0) \to X(G_2)$ be rearrangements, and let $c$ be the $1$-cube spanned by $[f]$ and $[g]$. Then the wall of $c$ is isomorphic to the complex $K(G_1 \setminus \varepsilon)$.
\end{lemma}

\begin{proof}
 Let $H$ be the wall of $c$ in $K(G_0)$, so the $0$-cubes of $H$ are the $1$-cubes of $K(G_0)$ parallel to $c$. Let $d$ be a such a $1$-cube of $H$. Say $d$ has $0$-faces $[f']$ and $[g']$, with $g' = \Delta_\varepsilon \circ f'$. Since $d$ is parallel to $c$, we know that $f'=\phi\circ f$ for some rearrangement $\phi \colon X(G_1)\to X(G')$, and that $\phi$ is a composition of expansions and contractions that never involve the edge $\varepsilon$.
 
 Now, with all the above data for $0$-cubes $d$ in $H$, define a map $\Psi \colon H^{(0)} \to K(G_1 \setminus \varepsilon)^{(0)}$ by
 \[
 \Psi(d) \defeq [\phi|_{K(G_1 \setminus \varepsilon)}] \text{.}
 \]
 Here $\phi|_{K(G_1 \setminus \varepsilon)}$ is, as the notation implies, the restriction of $\phi$ to a rearrangement $X(G_1\setminus \varepsilon)\to X(G'\setminus \varepsilon)$. This is well defined since $\phi$ is a composition of expansions and contractions that never involve $\varepsilon$. It is easily seen that $\Psi$ is bijective, and that it extends to a cubical isomorphism $\Psi\colon H\to K(G_1 \setminus \varepsilon)$.
\end{proof}

Our specific walls of interest arise as follows. Consider the family of graphs $J_n$ with $n+5$ vertices and $2n+10$ edges indicated in Figure~\ref{fig:center_vertex}.

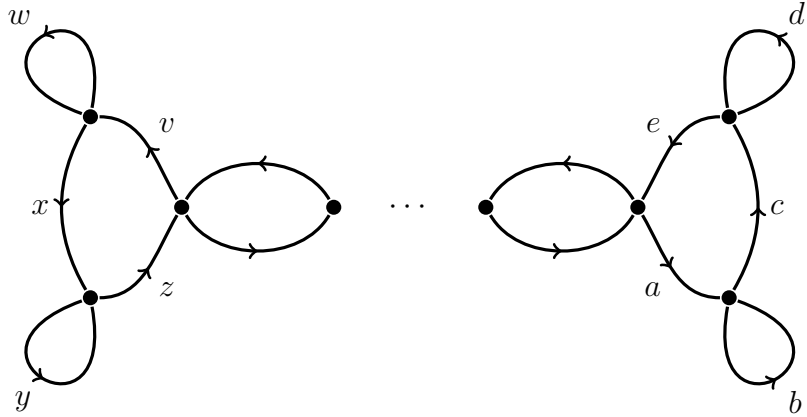
\begin{figure}[htb]
\begin{center}
\begin{tikzpicture}[scale=2, very thick]
\begin{scope}[every node/.style={circle, fill, inner sep=2pt}]
\node (oo) at (1,0) {};
\node (o) at (0,0) {};
\node (p) at (-.6,.6) {};
\node (q) at (-.6,-.6) {};
\node (kk) at (2,0) {};
\node (k) at (3,0) {};
\node (m) at (3.6,.6) {};
\node (l) at (3.6,-.6) {};
\end{scope}
\node at (1.5,0) () {$\cdots$};
\begin{scope}
\path (o) edge[->-,out=0-\picangle, in = 180+\picangle, below] (oo);
\path (oo) edge[->-,out=180-\picangle, in = 0+\picangle, below] (o);
\path (o) edge[->-,out=180-\picangle, in = 60-\picangle, above right] node {$v$} (p);
\path (p) edge[->-,out=120-\loopangle, in = 120+\loopangle, min distance = 1cm, looseness=5, above left] node {$w$} (p);
\path (p) edge[->-,out=180+\picangle, in = 180-\picangle, left] node {$x$} (q);
\path (q) edge[->-,out=240-\loopangle, in = 240+\loopangle, min distance = 1cm, looseness=5, below left] node {$y$} (q);
\path (q) edge[->-,out=300+\picangle, in = 180+\picangle, below right] node {$z$} (o);

\path (kk) edge[->-,out=0-\picangle, in = 180+\picangle, below] (k);
\path (k) edge[->-,out=180-\picangle, in = 0+\picangle, below] (kk);
\path (k) edge[->-,out=0-\picangle, in = 240-\picangle, below left] node {$a$} (l);
\path (l) edge[->-,out=300-\loopangle, in = 300+\loopangle, min distance = 1cm, looseness=5, below right] node {$b$} (l);
\path (l) edge[->-,out=0+\picangle, in = 0-\picangle, right] node {$c$} (m);
\path (m) edge[->-,out=60-\loopangle, in = 60+\loopangle, min distance = 1cm, looseness=5, above right] node {$d$} (m);
\path (m) edge[->-,out=120+\picangle, in = 0+\picangle, above left] node {$e$} (k);
\end{scope}
\end{tikzpicture}
\end{center}
\caption{The family of graphs $J_n, n \in \N$ with various edges labeled. The graph $J_n$ has $n+5$ vertices and $2n+10$ edges.}
\label{fig:center_vertex}
\end{figure}

Fix a vertex $[f_n]$ of $K(G_0)$ represented by a rearrangement $X(G_0)\to X(J_n)$. These exist because $J_n$ is an expansion of $G_0$. Let $c_n$ be the $2$-cube whose $0$-faces are $[f_n]$, $[\nabla_{c,d,e}\circ f_n]$, $[\nabla_{x,y,z}\circ f_n]$ and $[\nabla_{c,d,e}\circ \nabla_{x,y,z} \circ f_n]$. Let $H_1$ be the hyperplane that cuts through the $1$-cube from $[f_n]$ to $[\nabla_{x,y,z}\circ f_n]$ and $H_2$ the hyperplane that cuts through the $1$-cube from $x_n$ to $[\nabla_{c,d,e}\circ f_n]$. For $i=1,2$, let $H_i^+$ be the half-spaces containing $[f_n]$ and let $H_i^-$ be the half-spaces that do not contain $[f_n]$.

By Lemma~\ref{lem:model_wall}, the intersection $H_1\cap H_2$ is isomorphic to $K(O_n)$, where $O_n$ is the graph given in Figure~\ref{fig:O_n}.

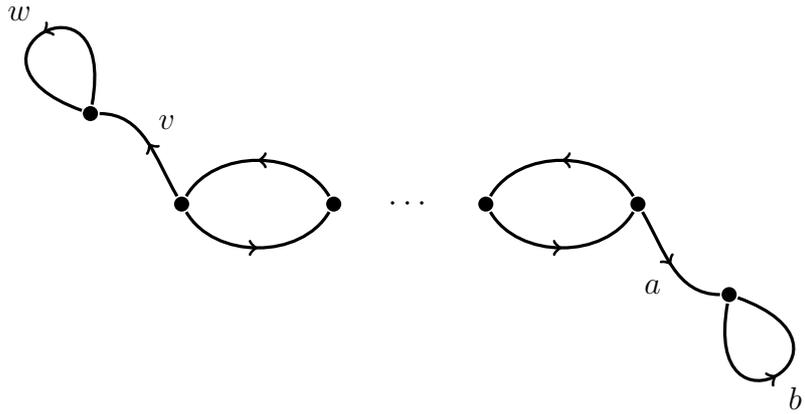
\begin{figure}[htb]
\begin{center}
\begin{tikzpicture}[scale=2, very thick]
\begin{scope}[every node/.style={circle, fill, inner sep=2pt}]
\node (oo) at (1,0) {};
\node (o) at (0,0) {};
\node (p) at (-.6,.6) {};
\node (kk) at (2,0) {};
\node (k) at (3,0) {};
\node (l) at (3.6,-.6) {};
\end{scope}
\node[label=below left :{$$}] (q) at (-.6,-.6) {};
\node (m) at (3.6,.6) {};
\node at (1.5,0) () {$\cdots$};
\begin{scope}
\path (o) edge[->-,out=0-\picangle, in = 180+\picangle, below] (oo);
\path (oo) edge[->-,out=180-\picangle, in = 0+\picangle, below] (o);
\path (o) edge[->-,out=180-\picangle, in = 60-\picangle, above right] node {$v$} (p);
\path (p) edge[->-,out=120-\loopangle, in = 120+\loopangle, min distance = 1cm, looseness=5, above left] node {$w$} (p);

\path (kk) edge[->-,out=0-\picangle, in = 180+\picangle, below] (k);
\path (k) edge[->-,out=180-\picangle, in = 0+\picangle, below] (kk);
\path (k) edge[->-,out=0-\picangle, in = 240-\picangle, below left] node {$a$} (l);
\path (l) edge[->-,out=300-\loopangle, in = 300+\loopangle, min distance = 1cm, looseness=5, below right] node {$b$} (l);
\end{scope}
\end{tikzpicture}
\end{center}
\caption{The family of graphs $O_n, n \in \N$. The graph $O_n$ has $n+3$ vertices and $2n+4$ edges.}
\label{fig:O_n}
\end{figure}

The following is our key technical lemma, and is the result that is most specific to the case of the Basilica rewriting system.

\begin{lemma}[Key technical lemma]\label{lem:empty_sublevel}
 Let $n \in \N$ be arbitrary. Let $O_n$ be the graph on $n + 3$ vertices and $2n+4$ edges indicated in Figure~\ref{fig:O_n}. Then the sublevel complex $K(O_n)_{2n+3}$ is empty.
\end{lemma}

\begin{proof}
 We need to show that any graph obtained from $O_n$ via a sequence of expansions and contractions has at least $2n+4$ edges. Call such a graph \emph{relevant}. Note first that for any relevant graph, denoting by $V$ the number of vertices and $E$ the number of edges, we have $E - (2n+4) = 2(V - (n+3))$. Thus we must equivalently show that any relevant graph has at least $n + 3$ vertices. Note further that in a relevant graph all edges lie in the boundary of the outer region, so they inherit a cyclic ordering. We call a path $e_1,\dots,e_k$ in a relevant graph \emph{ordered} if the edges are successive in the cyclic order.
 
 We say that a non-constant ordered closed path in a relevant graph $G$ is a \emph{special circuit} if it only meets vertices of degree $4$. We call a vertex $v$ of such a graph \emph{collapsible} if it is of degree $4$ and there is a special circuit connecting $v$ to itself. We see that $O_n$ has zero collapsible vertices. Intuitively, if $v$ is a collapsible vertex then after repeated contractions it can be made to disappear.
 
 Let $G$ be a relevant graph, let $\varepsilon$ be an edge of $G$ and let $G \lhd \varepsilon$ be the corresponding simple expansion. Then the new vertex of $G \lhd \varepsilon$ is collapsible: the special circuit in question is just the new loop, $\varepsilon2$. Any vertex of $G$ that is collapsible is still collapsible in $G \lhd \varepsilon$: a special circuit in $G$ that uses $\varepsilon$ can be turned into a special circuit in $G \lhd e$ by replacing $\varepsilon$ by $\varepsilon 1, \varepsilon 2, \varepsilon 3$. Finally, if a vertex $v$ of $G$ is not collapsible, then it does not become collapsible in $G \lhd \varepsilon$: indeed, a special circuit in $G \lhd \varepsilon$ at $v$ that uses one of $\varepsilon 1$, $\varepsilon 2$ or $\varepsilon 3$ has to use all of them, so replacing the sequence $\varepsilon 1, \varepsilon 2, \varepsilon 3$ by $\varepsilon$ gives rise to a special circuit in $G$.
 
 The result of this discussion is that the number of collapsible vertices minus the number of vertices is invariant under arbitrary expansions and contractions. Since the number of collapsible vertices cannot become negative, we conclude that any graph obtained from $O_n$ via expansions and contractions has to have at least $n+4$ vertices.
\end{proof}

Let $Y_n\defeq K(G_0)_{2n+9}$. We will now verify that for $n$ sufficiently large, all the hypotheses of Lemma~\ref{lem:two_walls} hold, with $Y_n$ playing the role of $Y$ and $H_1$ and $H_2$ being our two walls.

\begin{corollary}[Empty middle]\label{cor:empty_intersection}
 The intersection $H_1\cap H_2\cap Y_n$ is empty.
\end{corollary}

\begin{proof}
 By Lemma~\ref{lem:model_wall} the intersection $H_1\cap H_2$ is isomorphic to $K(O_n)$, and $H_1 \cap H_2 \cap Y_n$ is isomorphic to $K(O_n)_{2n+3}$ since a point in $K(O_n)$ with rank $r$ corresponds to a $2$-cube of $H_1\cap H_2$ with maximum rank $r+6$. By Lemma~\ref{lem:empty_sublevel} $K(O_n)_{2n+3}$ is empty.
\end{proof}

\begin{lemma}[Non-empty quarter-spaces]\label{lem:nonempty_quarterspaces}
 The intersections $H_1^{\delta_1}\cap H_2^{\delta_2} \cap Y_n$ are non-empty for any $\delta_1,\delta_2 \in\{+,-\}$.
\end{lemma}

\begin{proof}
 The vertices $[\nabla_{a,b,c} \circ \nabla_{v,w,x} \circ f_n]$, $[\nabla_{c,d,e} \circ \nabla_{v,w,x} \circ f_n]$, $[\nabla_{a,b,c} \circ \nabla_{x,y,z} \circ f_n]$ and $[\nabla_{c,d,e} \circ \nabla_{x,y,z} \circ f_n]$ respectively lie in $H_1^+\cap H_2^+ \cap Y_n$, $H_1^+\cap H_2^- \cap Y_n$, $H_1^-\cap H_2^+ \cap Y_n$ and $H_1^-\cap H_2^- \cap Y_n$, and each of these vertices has rank $2n+6$, so lies in $Y_n$.
\end{proof}

\begin{lemma}[Connected half-spaces]\label{lem:conn_halfspaces}
 Let $n \in \N$. For $i\in\{1,2\}$ and $\delta\in\{+,-\}$, we have that $H_i^\delta \cap Y_n$ is connected.
\end{lemma}

\begin{figure}[htb]
\begin{tikzpicture}[scale=1.7, very thick,none/.style={text opacity=0,opacity=0,fill opacity=0}]
\useasboundingbox (-1,-1) rectangle (3,1.25);
\begin{scope}[every node/.style={circle, fill, inner sep=2pt}]
\node[label=below:{}] (x) at (0,0) {};
\node[label=below:{}] (y) at (1,0) {};
\node[label=below:{}] (z) at (2,0) {};
\draw[none] (z) to [out=180-\picangle, in=0+\picangle]  coordinate[pos=0.5] (yz) (y);
\node[label=below:{}] (yz) at (yz) {};
\end{scope}
\begin{scope}
\path (x) edge[line width= 2pt, OliveGreen,->-,out=180-\loopangle, in = 180+\loopangle, min distance = 1cm, looseness=5, left] node {$\delta$} (x);
\path (z) edge[line width= 2pt, blue,->-,out=0-\loopangle, in = 0+\loopangle, min distance = 1cm, looseness=5, right] node {$B$} (z);

\path (z) edge[line width= 2pt, blue,->-,out=180-\picangle, in = 0, above right] node {} (yz);
\path (z) edge[red,->-,out=180-\picangle, in = 0, above right] node {} (yz);
\path (yz) edge[line width= 2pt, red, ->-,out=90-\loopangle, in = 90+\loopangle, min distance = 1cm, looseness=5, above] node {$A$} (yz);
\path (yz) edge[line width= 2pt, red, ->-,out=180, in = 0+\picangle, above left] node {} (y);

\path (x) edge[line width= 2pt, OliveGreen,->-,out=0-\picangle, in = 180+\picangle, below] node {$\varepsilon$} (y);
\path (y) edge[line width= 2pt, OliveGreen,->-,out=180-\picangle, in = 0+\picangle, above] node {$\gamma$} (x);

\path (y) edge[line width= 2pt, blue,->-,out=0-\picangle, in = 180+\picangle, below] node {} (z);
\node at (-.5,-0.5) {\textcolor{OliveGreen}{$Z$}};
\end{scope}
\end{tikzpicture}
\begin{tikzpicture}[scale=1.7, very thick,none/.style={text opacity=0,opacity=0,fill opacity=0}]
\useasboundingbox (-1,-1) rectangle (3,1.25);
\begin{scope}[every node/.style={circle, fill, inner sep=2pt}]
\node[label=below:{}] (x) at (0,0) {};
\node[label=below:{}] (y) at (1,0) {};
\node[label=below:{}] (z) at (2,0) {};
\draw[none] (z) to [out=180-\picangle, in=0+\picangle]  coordinate[pos=0.5] (yz) (y);
\node[label=below:{}] (yz) at (yz) {};
\draw[none] (x) to [out=0-\picangle, in = 180+\picangle, below] coordinate[pos=0.5] (xy) (y);
\node[label=below:{}] (xy) at (xy) {};
\end{scope}
\begin{scope}
\path (x) edge[->-,out=180-\loopangle, in = 180+\loopangle, min distance = 1cm, looseness=5, left] node {} (x);
\path (z) edge[line width= 2pt, blue,->-,out=0-\loopangle, in = 0+\loopangle, min distance = 1cm, looseness=5, right] node {$B$} (z);

\path (z) edge[line width= 2pt, blue,->-,out=180-\picangle, in = 0, above right] node {} (yz);
\path (z) edge[red,->-,out=180-\picangle, in = 0, above right] node {} (yz);
\path (yz) edge[line width= 2pt, red, ->-,out=90-\loopangle, in = 90+\loopangle, min distance = 1cm, looseness=5, above] node {$A$} (yz);
\path (yz) edge[line width= 2pt, red, ->-,out=180, in = 0+\picangle, above left] node {} (y);

\path (x) edge[line width = 2pt, Sepia,->-,out=0-\picangle, in = 180, below left] node {$\varepsilon 1$} (xy);
\path (xy) edge[line width = 2pt, Sepia,->-,out=-90-\loopangle, in = -90+\loopangle, min distance = 1cm, looseness=5, below] node {$\varepsilon 2$} (xy);
\path (xy) edge[line width = 2pt, Sepia,->-,out=0, in = 180+\picangle, below right] node {$\varepsilon 3$} (y);

\path (y) edge[->-,out=180-\picangle, in = 0+\picangle, below] node {} (x);

\path (y) edge[line width= 2pt, blue,->-,out=0-\picangle, in = 180+\picangle, below] node {} (z);
\end{scope}
\end{tikzpicture}

\begin{tikzpicture}[scale=1.7, very thick,none/.style={text opacity=0,opacity=0,fill opacity=0}]
\useasboundingbox (-1,-1) rectangle (3,1.25);
\begin{scope}[every node/.style={circle, fill, inner sep=2pt}]
\node[label=below:{}] (x) at (0,0) {};
\node[label=below:{}] (y) at (1,0) {};
\node[label=below:{}] (z) at (2,0) {};
\draw[none] (z) to [out=180-\picangle, in=0+\picangle]  coordinate[pos=0.5] (yz) (y);
\node[label=below:{}] (yz) at (yz) {};
\draw[none] (x) to [out=0-\picangle, in = 180+\picangle, below] coordinate[pos=0.5] (xy) (y);
\node[label=below:{}] (xy) at (xy) {};
\end{scope}
\begin{scope}
\path (x) edge[line width=2pt, DarkOrchid,->-,out=180-\loopangle, in = 180+\loopangle, min distance = 1cm, looseness=5, left] node {} (x);
\path (z) edge[line width= 2pt, blue,->-,out=0-\loopangle, in = 0+\loopangle, min distance = 1cm, looseness=5, right] node {$B$} (z);

\path (z) edge[line width= 2pt, blue,->-,out=180-\picangle, in = 0, above right] node {} (yz);
\path (z) edge[red,->-,out=180-\picangle, in = 0, above right] node {} (yz);
\path (yz) edge[line width= 2pt, red, ->-,out=90-\loopangle, in = 90+\loopangle, min distance = 1cm, looseness=5, above] node {$A$} (yz);
\path (yz) edge[line width= 2pt, red, ->-,out=180, in = 0+\picangle, above left] node {} (y);

\path (x) edge[line width=2pt, DarkOrchid,->-,out=0-\picangle, in = 180, below left] node {} (xy);
\path (xy) edge[->-,out=-90-\loopangle, in = -90+\loopangle, min distance = 1cm, looseness=5, below] node {} (xy);
\path (xy) edge[->-,out=0, in = 180+\picangle, below right] node {} (y);

\path (y) edge[line width=2pt, DarkOrchid,->-,out=180-\picangle, in = 0+\picangle, above] node {} (x);

\path (y) edge[line width= 2pt, blue,->-,out=0-\picangle, in = 180+\picangle, below] node {} (z);
\node at (-.5,-0.5) {\textcolor{DarkOrchid}{$Z'$}};
\end{scope}
\end{tikzpicture}
\begin{tikzpicture}[scale=1.7, very thick,none/.style={text opacity=0,opacity=0,fill opacity=0}]
\useasboundingbox (-1,-1) rectangle (3,1.25);
\begin{scope}[every node/.style={circle, fill, inner sep=2pt}]
\node[label=below:{}] (y) at (1,0) {};
\node[label=below:{}] (z) at (2,0) {};
\draw[none] (z) to [out=180-\picangle, in=0+\picangle]  coordinate[pos=0.5] (yz) (y);
\node[label=below:{}] (yz) at (yz) {};
\draw[none] (x) to [out=0-\picangle, in = 180+\picangle, below] coordinate[pos=0.5] (xy) (y);
\node[label=below:{}] (xy) at (xy) {};
\end{scope}
\begin{scope}
\path (z) edge[line width= 2pt, blue,->-,out=0-\loopangle, in = 0+\loopangle, min distance = 1cm, looseness=5, right] node {$B$} (z);

\path (z) edge[line width= 2pt, blue,->-,out=180-\picangle, in = 0, above right] node {} (yz);
\path (z) edge[red,->-,out=180-\picangle, in = 0, above right] node {} (yz);
\path (yz) edge[line width= 2pt, red, ->-,out=90-\loopangle, in = 90+\loopangle, min distance = 1cm, looseness=5, above] node {$A$} (yz);
\path (yz) edge[line width= 2pt, red, ->-,out=180, in = 0+\picangle, above left] node {} (y);

\draw[line width=2pt,Magenta,->-] (y) to[out=180-\picangle, in=90] (x) to[out=-90, in = 180] (xy);
\path (xy) edge[line width=2pt,Magenta,->-,out=-90-\loopangle, in = -90+\loopangle, min distance = 1cm, looseness=5, above] node {} (xy);
\path (xy) edge[line width=2pt,Magenta,->-,out=0, in = 180+\picangle, above left] node {} (y);


\path (y) edge[line width= 2pt, blue,->-,out=0-\picangle, in = 180+\picangle, below] node {} (z);
\node at (-.25,-0.5) {\textcolor{Magenta}{$Z''$}};
\end{scope}
\end{tikzpicture}
\caption{An example of the situation at the end of the proof of Lemma~\ref{lem:conn_halfspaces}.}
\label{fig:graph_dance}
\end{figure}
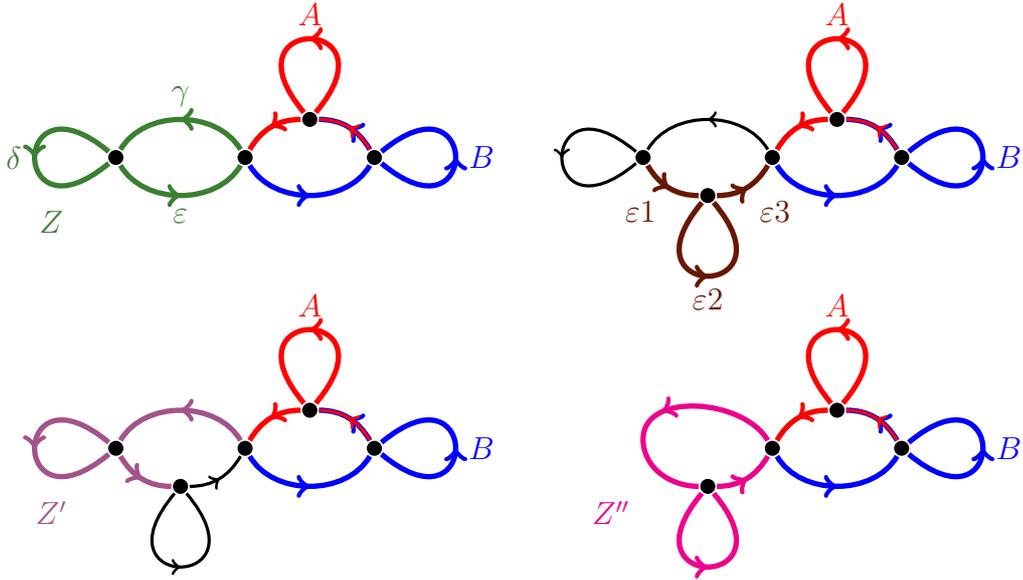

\begin{proof}
We know from \cite[Example~4.5]{belk15a} that the only $0$-cubes $x\in H_i^\delta$ with disconnected descending link have at most $4$ vertices. Thus if $h(x)> 2n+9$ the descending link in $K(G_0)$ is connected. Now $H_i^-$ is contractible (being convex) and for every vertex $x \in H_i^-$ the descending link in $H_i^-$ is the same as the descending link in $K(G_0)$. So the Morse argument shows that $H_i^- \cap Y_n$ is connected.
 
 Now we look at $H_i^+$. We take $i = 1$ for concreteness, but the case $i = 2$ is analogous. Let $x_1, x_2 \in (H_1^+ \cap Y_n)$ be $0$-cubes. Since $H_1^+$ is convex and thus contractible, there is an edge path from $x_1$ to $x_2$ in $H_1^+$. Our goal is to perturb it to lie in $Y_n$. For that purpose it suffices to show that given $0$-cubes $x, y, z \in H_1^+$ with $x$ adjacent to $y$ and $z$, $h(x) > 2n+9$, and $h(y), h(z) < h(x)$ there exists a path from $y$ to $z$ using only vertices of height at most $h(x)-1$. Let $x = [f]$ where $f$ is a rearrangement $X(G_0) \to X(G)$ and let $A$ and $B$ be the ($3$-edge) subgraphs of $G$ such that $y = [\nabla_A \circ f]$ and $z = [\nabla_B \circ f]$. There is a path from $y$ to $z$ in the descending link of $x$ in $K(G_0)$, but we need to deal with the fact that the descending link in $H_1^+$ may be smaller. We only need to treat the case where this actually happens, i.e.\ where some $1$-simplex over $x$ lies in $H_1$. Say the neighbor of $x$ in $H_1^-$ is $x' = [\nabla_Z \circ f]$ for a subgraph $Z$. Moreover, if the path from $y$ to $z$ in the descending link does not pass through $x'$, we are again done, so we can assume that it does pass through $x'$ and, in fact, that $y$ and $z$ are both adjacent to $x'$. In terms of graphs this means that $A$ and $B$ both are edge-disjoint from $Z$ but $A$ and $B$ have an edge in common.
 
 We need some further notation: let $A'$ be the image of $A$ after contracting $B$ and let $B'$ be the image of $B$ after contracting $A$. We informally describe the strategy before giving the technical details. The problem is that the path $y = [\nabla_A f], [\nabla_Z \nabla_A f], x'=[\nabla_Z f], [\nabla_Z \nabla_B f], z = [\nabla_B f]$ is not available because it does not lie in $H_1^+$. So rather than contracting $Z$ we would like to first split an edge $\varepsilon$ of $Z$ and then contract twice (preventing us from crossing $H_1$). The problem with this is that $h([\Delta_\varepsilon \nabla_A f]) = h(x)$. So before we can do that we need to contract $B'$ as well.
 
 To make this formal, we need even more notation: say $Z$ has edges $\gamma$, $\delta$, $\varepsilon$ in that order. So splitting $\varepsilon$ in $Z$ gives rise to a graph $Z \lhd \varepsilon$ with edges $\gamma$, $\delta$, $\varepsilon1$, $\varepsilon2$ and $\varepsilon3$ of which $\delta$ and $\varepsilon2$ are loops. We denote by $Z'$ the subgraph with edges $\gamma$, $\delta$, $\varepsilon1$ and by $Z''$ the result of contracting $Z'$ in $Z \lhd \varepsilon$. Then the vertices $[\Delta_\varepsilon f]$, $[\nabla_{Z'} \Delta_\varepsilon f]$, $[\nabla_{Z''} \nabla_{Z'} \Delta_\varepsilon f]$ all lie in $H_1^+$, but $h([\Delta_\varepsilon f]) = h(x) + 1$.
 
 Finally, here is the path from $y$ to $z$ that does not cross $H_1$ and only passes through vertices of height $< h(x)$. From $y$ it moves down to $[\nabla_{B'} \nabla_A f]$ of height $h(x)-2$. From there it moves through $[\Delta_\varepsilon \nabla_{B'} \nabla_A f]$, $[\nabla_{Z'} \Delta_\varepsilon \nabla_{B'} \nabla_A f]$ to $[\nabla_{Z''} \nabla_{Z'} \Delta_\varepsilon \nabla_{B'} \nabla_A f] = [\nabla_{B'} \nabla_A \nabla_{Z''} \nabla_{Z'} \Delta_\varepsilon f]$ of height $h(x) - 3$. At this point it can ``untwist'' $A$ and $B$ by moving through $[\nabla_A \nabla_{Z''} \nabla_{Z'} \Delta_\varepsilon f]$ to $[\nabla_{Z''} \nabla_{Z'} \Delta_\varepsilon f]$ (of height $h(x) - 1$) and on to $[\nabla_B \nabla_{Z''} \nabla_{Z'} \Delta_\varepsilon f]$ and $[\nabla_{A'} \nabla_B \nabla_{Z''} \nabla_{Z'} \Delta_\varepsilon f] = [\nabla_{Z''} \nabla_{Z'} \Delta_\varepsilon \nabla_{A'} \nabla_B f]$. Now proceeding symmetrically it passes through $[\nabla_{Z'} \Delta_\varepsilon \nabla_{A'} \nabla_B f]$, $[\Delta_\varepsilon \nabla_{A'} \nabla_B f]$, $[\nabla_{A'} \nabla_B f]$ to $z = [\nabla_B f]$.
 \end{proof}

We can now prove a strengthening of our main result.

\begin{theorem}
$T_B$ is not of type $\FP_2$, and hence is not finitely presented.
\end{theorem}

\begin{proof}
Using Lemma~\ref{lem:fin_iff_conn} it suffices to show that $Y_n$ is not $1$-acyclic for arbitrarily large $n$. This follows from Lemma~\ref{lem:two_walls} the hypotheses of which have been verified in Corollary~\ref{cor:empty_intersection} and Lemmas~\ref{lem:nonempty_quarterspaces} and~\ref{lem:conn_halfspaces}.
\end{proof}

\providecommand{\bysame}{\leavevmode\hbox to3em{\hrulefill}\thinspace}
\providecommand{\MR}{\relax\ifhmode\unskip\space\fi MR }
\providecommand{\MRhref}[2]{%
  \href{http://www.ams.org/mathscinet-getitem?mr=#1}{#2}
}
\providecommand{\href}[2]{#2}

\end{document}